\documentclass[a4paper,11pt,reqno]{amsart}
\usepackage{amsfonts}
\usepackage{amsmath}
\usepackage{logicproof}
\usepackage{amssymb, amsthm}
\usepackage{tabularx}
\usepackage{caption}
\usepackage{graphicx}
\usepackage{booktabs}
\usepackage{enumitem}
\usepackage{multirow}
\usepackage{array}
\newcolumntype{L}[1]{>{\raggedright\let\newline\\\arraybackslash\hspace{0pt}}m{#1}}
\newcolumntype{C}[1]{>{\centering\let\newline\\\arraybackslash\hspace{0pt}}m{#1}}
\newcolumntype{R}[1]{>{\raggedleft\let\newline\\\arraybackslash\hspace{0pt}}m{#1}}
\usepackage[symbol]{footmisc}
\usepackage[labelfont=bf,format=plain,justification=raggedright,singlelinecheck=false]{caption}
\usepackage{mathrsfs}
\usepackage[]{mdframed}
\usepackage[colorlinks]{hyperref}
\usepackage{tcolorbox}
\usepackage{adjustbox}
\tcbuselibrary{theorems}
\newtcbtheorem[number within=section]{mytheo}{Note}%
{colback=white!5,colframe=black!35!black,fonttitle=\bfseries}{th}
\numberwithin{equation}{section}

\usepackage{graphicx}
\usepackage{textgreek}

 {
      \theoremstyle{plain}
      \newtheorem{assumption}{Assumption}
  }
\newtheorem{theorem}{Theorem}[section]

\newtheorem{remark}[theorem]{Remark}

\newtheorem{lemma}[theorem]{Lemma}

\begin{document}
	
	\title[Nyström Subsampling for FLR model]{Convergence Analysis of regularised Nyström method for Functional Linear Regression}
\author[N. Gupta]{Naveen Gupta}
\address[N. Gupta]{Department of Mathematics, Indian Institute of Technology Delhi, India}
\email{ngupta.maths@gmail.com}
\author{S. Sivananthan}
\address[S. Sivananthan]{Department of Mathematics, Indian Institute of Technology Delhi, India}
\email{siva@maths.iitd.ac.in}
 
\begin{abstract}
The functional linear regression model has been widely studied and utilized for dealing with functional predictors. In this paper, we study the Nyström subsampling method, a strategy used to tackle the computational complexities inherent in big data analytics, especially within the domain of functional linear regression model in the framework of reproducing kernel Hilbert space. 

By adopting a Nyström subsampling strategy, our aim is to mitigate the computational overhead associated with kernel methods, which often struggle to scale gracefully with dataset size. Specifically, we investigate a regularization-based approach combined with Nyström subsampling for functional linear regression model, effectively reducing the computational complexity from $O(n^3)$ to $O(m^2 n)$, where $n$ represents the size of the observed empirical dataset and $m$ is the size of subsampled dataset. Notably, we establish that these methodologies will achieve optimal convergence rates, provided that the subsampling level is appropriately selected. We have also demonstrated a numerical example of Nyström subsampling in the RKHS framework for the functional linear regression model.
\end{abstract}

\keywords{Functional Linear regression, reproducing kernel Hilbert space, Nyström subsampling, regularization, Covariance operator}
\maketitle

\section{Introduction}
Functional Linear Regression (FLR) model, introduced by Ramsay and Dalzell \cite{ramsay1991some}, has been one of the main tool of functional data analysis which deals with the data given in the form of curves \cite{ ramsaywhendataarefunctions,ramsay2002afda,morris2015functional,kokoszka2017}. In the literature, two main techniques have been utilised to solve the FLR model. One is functional principal component analysis \cite{cardot1999, muller2005generalized, cai2006prediction}, which relies on representing the target function with respect to a basis. For example, a B-spline basis as in \cite{cardot2003spline}, or more popularly the eigenfunctions of the covariance operator \cite{cai2006prediction, hall2007methodology}. Another approach for the FLR model that gained popularity is the method of regularization in a reproducing kernel Hilbert space (RKHS), i.e., an estimator of the target function is constructed by restricting it to an RKHS \cite{ARKHSFORFLR, tonyyuan2012minimax, ZhangFaster2020, balasubramanian2024functional, gupta2024optimal}. As this method needs to solve a matrix equation with the kernel matrix of size $n \times n$, the computational cost of this method is of $O(n^3)$. Though this method is simple and has strong empirical performance, its high computational cost makes it inefficient for large dataset. To overcome this issue, we propose to utilise the Nyström subsampling method for the FLR model to reduce the computational complexity of the RKHS-based method. The pivotal step of the Nyström method involves approximating a large matrix by employing a smaller matrix obtained by randomly subsampling columns of the original matrix. Theoretical study of the Nyström approach within the classical learning theory framework has been conducted in various works \cite{bach2013sharp, sergei2013,  rudi2015less, rudi2017falkon, sergiyjrnystrom, sergiyjrgeneralsmoothness, sergiylowsmoothness, KCGMwithrandomprojection}. Other than Nyström method, there has been a number of techniques in the literature to deal with the computational complexity issues such as Distributed learning \cite{zhou2017distributed,tong2022distributed, liu2024statistical}, greedy method \cite{smola2000sparse} and so on. \\

FLR model extends the classical linear regression to functional data, where predictors and/or responses are functions rather than scalar values. We put our interest in the model where predictors are functions and responses are scalar values. Mathematically, the FLR model with scalar responses can be delineated as follows:
Consider a real-valued random variable denoted as $Y(\omega)$, and let $(X(\omega,t); t\in S , \omega \in \Omega)$ represent a square-integrable continuous-time process defined on the same space. In this context, the FLR model is given as follows:
\begin{equation}
\label{model}
    Y = \int_{S} X(t) \beta^*(t)\, dt + \epsilon.
\end{equation}
Here, $\epsilon$ denotes a zero-mean random noise that is independent of $X$, possessing a finite variance $\sigma^2$, $\beta^*$ represents an unknown square integrable slope function and $S$ denotes a compact subset of $\mathbb{R}^d$. We also assume that $\mathbb{E}[\|X\|_{L^2}^2]$ is finite. 

It can be easily obtained from $(\ref{model})$ that for every $s \in S$, we have $(C\beta^*)(s) = \mathbb{E}[XY](s)$, where $C:=\mathbb{E}[X\otimes X]$ is the covariance operator. Alternatively, the slope function can be seen as: 
\begin{equation*}
    \label{truesolution}
    \beta^* := \arg\min_{\beta \in L^2(S)}\mathbb{E}[Y-\langle X,\beta \rangle_{L^2} ]^2.
\end{equation*}
As the slope function $\beta^*$ is not known, the aim is to construct an estimator $\hat{\beta}$  that approximates the slope function $\beta^*$ using the observed empirical data $\{(X_1, Y_1),(X_2, Y_2),\cdots, (X_n, Y_n)\}$. Here, the $X_i$'s denote independent and identically distributed (i.i.d.) samples of the random process $X(\omega,\cdot)$, and the $Y_{i}$'s represent scalar responses.

A recent study of kernel methods to solve the FLR model has been performed in \cite{balasubramanian2024functional}. It can be seen that practical use of the kernel method requires to inverse of an $n \times n$ size matrix, which has a complexity of $O(n^3)$. As a result, applying kernel methods becomes challenging when the dataset is large. Now the question arises whether can we use the Nyström approach to attain a complexity that is less than cubic in data size $n$ without sacrificing our rates of learning. 

To answer this question, we demonstrate how the regularized Nyström subsampling can be applied in the context of the FLR model. We derive the convergence rates for regularized Nyström method assuming that $\beta^* \in \mathcal{R}(T^{\frac{1}{2}}(T^{\frac{1}{2}}CT^{\frac{1}{2}})^s)$, where $T$ is the kernel integral operator and $\mathcal{R}(A)$ denotes range of an operator $A$. This assumption of smoothness aligns with standard practices in analyzing kernel methods for the FLR model \cite{ZhangFaster2020, balasubramanian2024functional}.

\subsection{Related work} The Nyström method is a well-established technique for approximating large kernel matrices in machine learning applications such as Support Vector Machines and Gaussian Processes. Williams and Seeger \cite{williams2000using} applied it to enhance the computational efficiency of Gram matrix approximations. Various sampling strategies for Nyström subsampling have been explored in \cite{drineas2005nys, sanjiv2012samplingmethod, rudi2015less, gittens2016nys}. Theoretical insights into Nyström subsampling within the framework of classical learning theory have been presented in \cite{rudi2015less, sergiyjrnystrom, sergiyjrgeneralsmoothness, sergiylowsmoothness}. Recent advancements in the method include its use in coefficient-based regularized regression \cite{coefficient2019based} and pairwise learning problems \cite{pairwise2023}.

\subsection{Contributions}
\begin{itemize}[label={}]
    \item \textbf{Computational efficiency:} In the RKHS framework, we propose an estimator for the FLR model based on regularised Nyström subsampling method to reduce the computational complexity of kernel method from $O(n^3)$ to $O(m^2n)$, where $n$ is the size of the original data and $m$ is the size of the subsampled data.\vspace{1.5mm}

    \item \textbf{Optimal rates:} Under the smoothness assumption $\beta^* \in \mathcal{R}(T^{\frac{1}{2}}(T^{\frac{1}{2}}CT^{\frac{1}{2}})^s)~0\leq s \leq \frac{1}{2}$, we provide convergence rates of Nyström method for estimation and prediction error in Theorem \ref{mainresult} which matches with the existing optimal rates (See Remark~\ref{combinedrmk}).
\end{itemize}



\noindent 
The structure of this paper unfolds as follows: In Section \ref{preliminaries}, we explore the framework of RKHS within the domain of the functional linear regression model. Here, we present the estimator of the slope function tailored for the Nyström method in the context of FLR. In section \ref{mr}, we provide the necessary assumptions that are crucial to our research. Next, we state the major outcomes and their proofs. Section \ref{supplements} provides essential results that serve as stepping stones for the proof of the main findings. Finally in section \ref{numerical}, we demonstrate a numerical example of Nyström subsampling method in the RKHS framework for the FLR model. 

\subsection{Notations} $L^2(S)$ denotes the space of all real-valued square-integrable functions defined on $S$. For $f, g \in L^2(S)$, $L^2$ inner product and norm are defined as $\langle f, g \rangle_{L^2} = \int_{S}f(x)g(x)\, dx$ and $\|f\|^2_{L^2}= \langle f, f \rangle_{L^2}$. The inner product and norm associated with the RKHS, $\mathcal{H}$ are denoted as $\langle \cdot, \cdot \rangle_{\mathcal{H}}$ and $\|\cdot\|_{\mathcal{H}}$, respectively. For two positive numbers $a$ and $b$, $a \lesssim b$ means $a \leq cb $ for some positive constant $c$. For a random variable $W$ with law $P$ and a constant $b$, $W\lesssim_p b$ denotes that for
any $\delta > 0$, there exists a positive constant $c_\delta<\infty$  such that $P(W\le c_\delta b)\ge \delta$. For notational convenience, we define $A := J^*CJ, \hat{A}_{n}:= J^*\hat{C}_{n}J$, $\Lambda:= T^{\frac{1}{2}}CT^{\frac{1}{2}}$ and $\hat{\Lambda}_{n}:= T^{\frac{1}{2}}\hat{C}_{n}T^{\frac{1}{2}}$, where $\hat{C}_n$ is an empirical estimator of $C$ (see Section~\ref{preliminaries} for the definition of $\hat{C}_{n}$). $I$ and $I_{t}$ are identity operators on spaces $\mathcal{H}$ and $\mathbb{R}^{t}$ respectively. For an operator $B: \mathcal{H}_1 \to \mathcal{H}_2$, $\|B\|_{\mathcal{H}_1 \to \mathcal{H}_2}$ represents the operator norm where $\mathcal{H}_1$ and $\mathcal{H}_2$ are two Hilbert spaces.

\section{Preliminaries}
\label{preliminaries}
An RKHS is a Hilbert space $\mathcal{H}$ of real-valued functions on $S$, for which the point-wise evaluation map $(f \to f(x))$ is linear and continuous for each $x \in S$. Associated to each RKHS, we have a reproducing kernel $k$ which is a symmetric and positive definite function from $S\times S$ to $\mathbb{R}$ such that $k(s,\cdot) \in \mathcal{H}$ and $f(s) = \langle k(s,\cdot),f \rangle_{\mathcal{H}}, \,\, \forall f \in \mathcal{H}$. We assume that $k$ is measurable and $\sup_{x\in S}k(x,x)\leq \kappa $, where $\kappa $ is a positive constant. Then the elements in the associated RKHS $\mathcal{H}$ are measurable and bounded functions on $S$. For a more comprehensive exploration of RKHS, we refer the reader to \cite{aronszajn1950rkhs, SVM2008steinwart, paulsen2016rkhs, AI2022sergei}.

The goal of our analysis is to construct an estimator for the unknown slope function by utilizing the observed information in the form of empirical data points $\{(X_{i}, Y_{i})\}_{i=1}^{n}$. For regularization parameter $\lambda>0$, the RKHS-based linear least square estimator is given by solving a minimization problem over the RKHS $\mathcal{H}$:
\begin{equation}
\label{estimatorequation}
    \begin{split}
        \hat{\beta}_{n} & = \arg\min_{\beta \in \mathcal{H}}\frac{1}{n}\sum_{i=1}^n[Y_i-\langle \beta, X_i \rangle_{L^2} ]^2 + \lambda \|\beta\|_{\mathcal{H}}^2 \\
        & = \arg\min_{\beta \in \mathcal{H}}\frac{1}{n}\sum_{i=1}^n[Y_i-\langle J\beta, X_i \rangle_{L^2} ]^2 + \lambda \|\beta\|_{\mathcal{H}}^2\\
        & = \arg\min_{\beta \in \mathcal{H}}\frac{1}{n}\sum_{i=1}^n[Y_i-\langle \beta, J^*X_i \rangle_{L^2} ]^2 + \lambda \|\beta\|_{\mathcal{H}}^2.\\
    \end{split}
\end{equation}
Here $J$ is an inclusion operator which embeds the RKHS $\mathcal{H}$ into the space of square integrable functions $L^2$, $f \mapsto \langle k(\cdot, \cdot), f \rangle_{\mathcal{H}} $\cite{cuckerzhou2007learningtheory}. We define the integral operator $T:= JJ^* : L^2(S) \to L^2(S)$, where $J^*: L^2(S) \mapsto \mathcal{H}$ given as:
$$J^*f = \int_{S}f(s)k(s,\cdot)ds ,\quad x \in S $$ is the adjoint of the inclusion operator.\\
It is easy to verify that the solution of the optimization equation $(\ref{estimatorequation})$ can be given by solving the operator equation
\begin{equation}
    \label{estimator}
   (J^*\hat{C}_{n}J + \lambda I)\hat{\beta}_{n} = J^*\hat{R},
\end{equation}
where $$\hat{C}_{n} := \frac{1}{n}\sum_{i=1}^{n} X_{i} \otimes X_{i} \quad \text{ and } \quad \hat{R} := \frac{1}{n} \sum_{i=1}^{n} Y_{i}X_{i}. $$
Here $\otimes$ is $L^2$ tensor product defined as $f \otimes g (\cdot) = \langle f, \cdot\rangle_{L^2} g$. Without regularization, equation $(\ref{estimator})$ can be seen as a discretized form of $J^*CJ\beta^* = J^*\mathbb{E}[XY]$, which is an ill-conditioned inverse problem. Regularization techniques have been well explored in the learning theory framework to solve such ill-posed problems \cite{sergeionregularization, lin2020optimalspectral, smale2002foundation}. A comprehensive investigation of the Tikhonov regularization method for solving equation $(\ref{estimator})$ has been conducted in the following references \cite{tonyyuan2012minimax, HTONGFLR, TONGHUBER, ZhangFaster2020,balasubramanian2024functional}.\\
The representer theorem \cite{smale2002foundation, generaliserepresenter} ensures that the solution of the equation $(\ref{estimatorequation})$ can be written in the linear span of $\{\int_{S}k(\cdot, t)X_{i}(t)dt: 1\leq i \leq n\}$, i.e., there exists a vector $\textbf{a} = (a_{1},\cdots, a_{n})^{\top} \in \mathbb{R}^n$ such that 
$$\hat{\beta}_{n} = \sum_{i=1}^{n} a_{i} \int_{S}k(\cdot, t)X_{i}(t)dt. $$ Using this representation of $\hat{\beta}_{n}$ in $(\ref{estimatorequation})$, we get that 
\begin{equation}
\label{reprtforkernel}
\textbf{a} = (\textbf{K}+\lambda I)^{-1}Y,
\end{equation}
where $\textbf{K}$ is a matrix of size $n \times n$ with entries $[\textbf{K}]_{ij} = \int_{S}\int_{S} k(s,t) X_{i}(s) X_{j}(t) ds dt $ and $Y = (Y_1, Y_2, \cdots, Y_n)^{\top} \in \mathbb{R}^n$.\\ 
It is evident that calculating vector $\textbf{a}$ will require at least $O(n^3)$ operations, and when our dataset is substantial, the kernel method encounters this memory constraint. This limitation becomes particularly pronounced as the size of the dataset increases, making it challenging to process and analyze large volumes of data efficiently. Thus, alternative approaches, such as subsampling methods, become increasingly attractive for managing computational complexities associated with extensive datasets. Next, we provide the representation of the estimator for the slope function using the Nyström subsampling method in the RKHS framework.

\subsection{Nyström subsampling estimator}
The Nyström subsampling method can be seen as a restriction of the minimization problem $(\ref{estimatorequation})$ to a smaller space $\mathcal{H}_{m}$ defined as:
\begin{equation*}
    \mathcal{H}_{m} := \textit{span}\left\{\int_{S}k(\cdot, t)\tilde{X}_{i}(t)dt: 1\leq i \leq m\right\},
\end{equation*}
where $\{\tilde{X}_{1}, \tilde{X}_{2}, \cdots, \tilde{X}_{m}\}$ is a uniformly subsampled dataset from the original dataset  $\{X_{1}, X_{2}, \cdots, X_{n}\}$.
So for the Nystöm approach, the estimator for the slope function is given as:
\begin{equation}
\label{optimization over subsampled space}
    \begin{split}
        \hat{\beta}_{m}: & = \arg\min_{\beta \in \mathcal{H}_{m}}\frac{1}{n}\sum_{i=1}^n[Y_i-\langle \beta, X_i \rangle_{L^2} ]^2 + \lambda \|\beta\|_{\mathcal{H}}^2 .
    \end{split}
\end{equation}
The principal task of the Nyström subsampling method is to replace the matrix $\textbf{K}$ with a low-rank matrix, which is obtained by random subsampling of columns of matrix $\mathbf{K}$.
Let us define $P_{m}$ as a projection operator from $\mathcal{H}$ and with the range $\mathcal{H}_{m}$.

The representer theorem ensures that their exists a vector $\tilde{\textbf{a}} = (\tilde{a}_{1}, \cdots, \tilde{a}_{m})^{\top} \in \mathbb{R}^m$ such that $\hat{\beta}_{m} = \sum_{i=1}^{m} \tilde{a}_{i} \int_{S} k(t, \cdot) \tilde{X}_{i}(t) dt$.
Putting things back to $(\ref{optimization over subsampled space})$ and solving for $\tilde{\textbf{a}}$ gives us
$$(\tilde{a}_{1}, \cdots, \tilde{a}_{m})^{\top} = \tilde{\textbf{a}} = (\textbf{K}_{n m}^{\top}\textbf{K}_{nm} + \lambda n \textbf{K}_{mm})^{-1}\textbf{K}_{nm}^{\top}Y,$$
where $\textbf{K}_{n m}$  and $\textbf{K}_{mm}$ are  matrices of size $n \times m$ and $m \times m$ respectively with entries
$$ [\textbf{K}_{n m}]_{ij} = \int_{S}\int_{S} k(s,t) X_{i}(s) \tilde{X}_{j}(t) ds dt , \quad 1 \leq i \leq n \text{ and } 1 \leq j \leq m, $$
$$ [\textbf{K}_{m m}]_{ij} = \int_{S}\int_{S} k(s,t) \tilde{X}_{i}(s) \tilde{X}_{j}(t) ds dt , \quad 1 \leq i \leq m \text{ and } 1 \leq j \leq m, $$
$\textbf{K}_{n m}^{\top}$ is the transpose of the matrix $\textbf{K}_{n m}$
and $Y = (Y_1, Y_2, \cdots, Y_n)^{\top} \in \mathbb{R}^n$.\\

\noindent
As we take $m << n$, the cost to compute $\tilde{\textbf{a}}$ is $O(m^2 n)$ which is much less than $O(n^3)$ required to compute $\textbf{a} = (\textbf{K}+\lambda I)^{-1}Y$ in equation $(\ref{reprtforkernel})$. As the computational cost is reduced, the question that needs to be answered is whether the accuracy of the kernel method is compromised or not which will be answered by Theorem \ref{mainresult} in section \ref{mr}.\\

By observing that any element of the RKHS $\beta \in \mathcal{H}$ can be written as sum of two elements  $\beta_{1} \in \mathcal{H}_{m}$, $\beta_{2} \in \mathcal{H}_{m}^{\perp}$ and $P_{m}\beta_{2}$ = 0, we can see that the solution of $(\ref{optimization over subsampled space})$ can be given as:
\begin{equation}
\label{Nystromoptimization}
   \hat{\beta}_{m} =  \arg\min_{\beta \in \mathcal{H}}\frac{1}{n}\sum_{i=1}^n[Y_i-\langle P_{m}\beta, X_i \rangle_{L^2} ]^2 + \lambda \|\beta\|_{\mathcal{H}}^2 .
\end{equation}
The solution of the minimization equation in $(\ref{Nystromoptimization})$ can be given by solving the operator equation:
\begin{equation}
\label{operatoreq}
    (P_{m}J^*\hat{C}_{n}JP_{m} + \lambda I)\hat{\beta}_{m}= P_{m}J^*\hat{R}.
\end{equation}

Let $P_{m} = VV^*$ such that $V^*V = I_{t} : \mathbb{R}^t \to \mathbb{R}^t$, where $V$ is an operator from $\mathbb{R}^t$ to $\mathcal{H}$ and $V^*: \mathcal{H} \to \mathbb{R}^t$ is the adjoint of operator $V$. Here $t$ is the dimension of the space $\mathcal{H}_{m}$.\\
Thus by using this decomposition of projection $P_{m}$ in $(\ref{operatoreq})$, it is easy to deduce that
\begin{equation*}
    \hat{\beta}_{m}= V(V^*\hat{A}_{n}V+\lambda I_{t})^{-1}V^*J^*\hat{R}.
\end{equation*}

\section{Main Results}
\label{mr}
In this section, we present the convergence analysis of our algorithm. First, we begin with a list of assumptions required for the analysis.

\begin{assumption}\label{as:1}
There exists an $h \in L^2(S)$ such that $\beta^* = T^{\frac{1}{2}}(T^{\frac{1}{2}}CT^{\frac{1}{2}})^sh$, where $0 \leq s \leq \frac{1}{2}$, i.e., $\beta^* \in \mathcal{R}\left(T^{\frac{1}{2}}(T^{\frac{1}{2}}CT^{\frac{1}{2}})^s\right)$.
\end{assumption}
In this analysis, we employ the Hölder-type source condition to put smoothness on the target function. As we know $\beta^* \in \mathcal{R}(T^{\frac{1}{2}}) = \mathcal{H}$, the assertion $\beta^* \in \mathcal{R}\left(T^{\frac{1}{2}}(T^{\frac{1}{2}}CT^{\frac{1}{2}})^s\right)$ exhibits a higher degree of smoothness compared to assuming $\beta^* \in \mathcal{H}$. 

\begin{assumption}\label{as:2}
     For some $b>1$,
\begin{equation*}
    i^{-b}\lesssim \tau_{i} \lesssim i^{-b} \quad \forall ~ i \in \mathbb{N},
\end{equation*}
where $(\tau_{i},e_{i})_{i\in \mathbb{N}}$ is the eigenvalue-eigenvector pair of operator $T^{\frac{1}{2}}CT^{\frac{1}{2}}$.
\end{assumption}
In Assumption \ref{as:2}, we demand the decay of the eigenvalues of the operator $T^{\frac{1}{2}}CT^{\frac{1}{2}} = \Lambda$. This condition ensures the bound on the effective dimension of the operator $\Lambda$ defined as 
$\mathcal{N}(\lambda)= \text{trace}(\Lambda (\Lambda+\lambda I)^{-1})$.\\
A quick idea of proof can be seen as:
\begin{equation}
\label{effectivedimention}
    \mathcal{N}(\lambda)= \sum_{i \in \mathbb{N}}\frac{\tau_{i}}{\tau_{i}+\lambda} \lesssim \sum_{i\in \mathbb{N}}\frac{i^{-b}}{i^{-b}+\lambda} \lesssim \lambda^{-\frac{1}{b}},
\end{equation}
where the last step follows from Lemma \ref{seriessum}.
\begin{assumption}\label{as:3}
\emph{(i)} There exists a constant $d>0$ such that
\begin{equation*}
    \mathbb{E}[\langle X, f\rangle^4] \leq d \mathbb{E}[\langle X, f\rangle^2]^2, \quad \forall ~ f \in L^2(S).
\end{equation*}
$$\text{Or}$$
\emph{(ii)} $$\sup_{\omega \in \Omega}\|X(\cdot, \omega)\|_{L^2} < \infty.$$
\end{assumption}
These two hypotheses differ because they apply to different types of stochastic processes. A non-degenerate Gaussian process will satisfy part (i), due to its well-behaved moments, but not part (ii), which requires boundedness. Conversely, a bounded process will satisfy part (ii) but may fail to meet the conditions of part (i), as it might not have the required moment properties that hold for Gaussian processes.

\begin{theorem}
\label{mainresult}
Suppose $\beta^*$ satisfies Assumption \ref{as:1}, i.e., $\beta^* = T^{\frac{1}{2}}(T^{\frac{1}{2}}CT^{\frac{1}{2}})^sh$ for some $h \in L^2(S)$ and $s \in [0,\frac{1}{2}]$. Further assume that Assumption \ref{as:2} and Assumption \ref{as:3} hold.
Then, for $\lambda = n^{-\frac{b}{1+b+2sb}}$, with probability at least $1-\delta$, we have
\begin{enumerate}[label=(\alph*)]
    \item\label{a} $\|C^{\frac{1}{2}}(\beta^*-\hat{\beta}_{m})\|_{L^2} \lesssim_{p} n^{-\frac{b(1+2s)}{2(1+b+2sb)}};$
    \item\label{b} $\|\beta^*-\hat{\beta}_{m}\|_{\mathcal{H}} \lesssim_{p} n^{-\frac{bs}{1+b+2sb}}$,
\end{enumerate}
where size of subsampled data is chosen such that $\lambda \gtrsim m^{-\frac{1}{b}} \geq n^{-\frac{1}{b}}$.
\end{theorem}
\subsection{Proof of part \ref{a}}
\begin{proof}
    We start by considering the error term
    \begin{equation*}
    \begin{split}
        \|C^{\frac{1}{2}}(\beta^* - \hat{\beta}_{m})\|_{L^2} = &  \|C^{\frac{1}{2}}(J\hat{\beta}_{m}-\beta^*)\|_{L^2} \\
        = & \|C^{\frac{1}{2}}(JV(V^*\hat{A}_{n}V+\lambda I_{t})^{-1}V^*J^*\hat{R}-\beta^*)\|_{L^2} \\
         \leq &  \underbrace{\|C^{\frac{1}{2}}JV(V^*\hat{A}_{n}V+\lambda I_{t})^{-1}V^*(J^*\hat{R}-J^*\hat{C}_{n}\beta^*)\|_{L^2}}_{\textit{Term-1}} \\
        & + \underbrace{\|C^{\frac{1}{2}}JV(V^*\hat{A}_{n}V+\lambda I_{t})^{-1}V^*J^*\hat{C}_{n}\beta^*-C^{\frac{1}{2}}\beta^*\|_{L^2}}_{\textit{Term-2}}.
    \end{split}
    \end{equation*}
Estimation of \textit{Term-1}:
\begin{equation*}
    \begin{split}
        & \|C^{\frac{1}{2}}JV(V^*\hat{A}_{n}V+\lambda I_{t})^{-1}V^*(J^*\hat{R}-J^*\hat{C}_{n}\beta^*)\|_{L^2}\\ 
        = & \|A^{\frac{1}{2}}V(V^*\hat{A}_{n}V+\lambda I_{t})^{-1}V^*(J^*\hat{R}-J^*\hat{C}_{n}\beta^*)\|_{\mathcal{H}}\\
        \leq &  \|A^{\frac{1}{2}}V(V^*\hat{A}_{n}V+\lambda I_{t})^{-1}V^*(A+\lambda I)^{\frac{1}{2}}\|_{\mathcal{H} \to \mathcal{H}}\|(A+\lambda I)^{-\frac{1}{2}}(J^*\hat{R}-J^*\hat{C}_{n}\beta^*)\|_{\mathcal{H}}.
    \end{split}
\end{equation*}
Using Lemma \ref{powersbound} ensures that
\begin{equation*}
    \begin{split}
        & \|C^{\frac{1}{2}}JV(V^*\hat{A}_{n}V+\lambda I_{t})^{-1}V^*(J^*\hat{R}-J^*\hat{C}_{n}\beta^*)\|_{L^2}\\
        \lesssim_{p} & \|(\hat{A}_{n}+\lambda I)^{\frac{1}{2}}V(V^*\hat{A}_{n}V+\lambda I_{t})^{-1}V^*(\hat{A}_{n}+\lambda I)^{\frac{1}{2}}\|_{\mathcal{H} \to \mathcal{H}}\|(A+\lambda I)^{-\frac{1}{2}}(J^*\hat{R}-J^*\hat{C}_{n}\beta^*)\|_{\mathcal{H}}\\
        \leq & \|(A+\lambda I)^{-\frac{1}{2}}(J^*\hat{R}-J^*\hat{C}_{n}\beta^*)\|_{\mathcal{H}},
    \end{split}
\end{equation*}
where the last step follows from Lemma \ref{lessthan1}. Now from Lemma \ref{Empirical bound} and $(\ref{effectivedimention})$, we have
\begin{equation*}
\begin{split}
    \|C^{\frac{1}{2}}JV(V^*\hat{A}_{n}V+\lambda I_{t})^{-1}V^*(J^*\hat{R}-J^*\hat{C}_{n}\beta^*)\|_{L^2} \lesssim_{p} & \sqrt{\frac{ \sigma^2 \mathcal{N}(\lambda)}{n \delta}}
    \lesssim  \frac{\lambda^{-\frac{1}{2b}}}{\sqrt{n}}.
    \end{split}
\end{equation*}

\noindent
Estimation of \textit{Term-2}:
\begin{equation*}
    \begin{split}
        & \|C^{\frac{1}{2}}JV(V^*\hat{A}_{n}V+\lambda I_{t})^{-1}V^*J^*\hat{C}_{n}\beta^*-C^{\frac{1}{2}}\beta^*\|_{L^2}\\
        \leq & \underbrace{\|C^{\frac{1}{2}}J(V(V^*\hat{A}_{n}V+\lambda I_{t})^{-1}V^*-(\hat{A}_{n}+\lambda I)^{-1})J^*\hat{C}_{n}\beta^*\|_{L^2}}_{\textit{Term-2a}}\\
        & \qquad \qquad + \underbrace{\|C^{\frac{1}{2}}J(\hat{A}_{n}+\lambda I)^{-1}J^*\hat{C}_{n}\beta^*-C^{\frac{1}{2}}\beta^*\|_{L^2}}_{\textit{Term-2b}}.
    \end{split}
\end{equation*}
Estimation of \textit{Term-2a}:
\begin{equation}
\label{term2a}
    \begin{split}
        & \|C^{\frac{1}{2}}J(V(V^*\hat{A}_{n}V+\lambda I_{t})^{-1}V^*-(\hat{A}_{n}+\lambda I)^{-1})J^*\hat{C}_{n}\beta^*\|_{L^2} \\
        = & \|A^{\frac{1}{2}}(V(V^*\hat{A}_{n}V+\lambda I_{t})^{-1}V^*-(\hat{A}_{n}+\lambda I)^{-1})J^*\hat{C}_{n}T^{\frac{1}{2}}(T^{\frac{1}{2}}CT^{\frac{1}{2}})^s h\|_{\mathcal{H}} \\
        \lesssim & \|A^{\frac{1}{2}}(V(V^*\hat{A}_{n}V+\lambda I_{t})^{-1}V^*-(\hat{A}_{n}+\lambda I)^{-1})J^*\hat{C}_{n}T^{\frac{1}{2}}\Lambda^s \|_{L^2 \to \mathcal{H}} \\
        \leq & \|A^{\frac{1}{2}}(V(V^*\hat{A}_{n}V+\lambda I_{t})^{-1}V^*-(\hat{A}_{n}+\lambda I)^{-1})J^*\hat{C}_{n}T^{\frac{1}{2}}(\Lambda + \lambda I)^{\frac{1}{2}}\|_{L^2 \to \mathcal{H}}\\
        & \qquad \qquad \qquad  \times\|(\Lambda + \lambda I)^{-\frac{1}{2}}\Lambda^s\|_{L^2 \to L^2}\\
        \stackrel{(*)}{=} & \|A^{\frac{1}{2}}(V(V^*\hat{A}_{n}V+\lambda I_{t})^{-1}V^*(\hat{A}_{n}+\lambda I)-I)(\hat{A}_{n}+\lambda I)^{-1}J^*\hat{C}_{n}J(A + \lambda I)^{\frac{1}{2}}\|_{\mathcal{H} \to \mathcal{H}}\\
        & \qquad \qquad \qquad  \times\|(\Lambda + \lambda I)^{-\frac{1}{2}}\Lambda^s\|_{L^2 \to L^2}\\
        \stackrel{(\dag)}{\lesssim_{p}} & \|A^{\frac{1}{2}}[I-VV^*+V(V^*\hat{A}_{n}V+\lambda I_{t})^{-1}V^*(\hat{A}_{n}+\lambda I)(VV^*-I)](\hat{A}_{n}+\lambda I)^{\frac{1}{2}}\|_{\mathcal{H} \to \mathcal{H}}\\
        & \qquad \qquad \qquad \times \|(\Lambda + \lambda I)^{-\frac{1}{2}}\Lambda^s\|_{L^2 \to L^2}\\
        \leq & \left\{\|A^{\frac{1}{2}}V(V^*\hat{A}_{n}V+\lambda I_{t})^{-1}V^*(\hat{A}_{n}+\lambda I)(VV^*-I)(\hat{A}_{n}+\lambda I)^{\frac{1}{2}}\|_{\mathcal{H} \to \mathcal{H}}\right.\\
        & \left.+ \|A^{\frac{1}{2}}(I-VV^*)(\hat{A}_{n}+\lambda I)^{\frac{1}{2}}\|_{\mathcal{H} \to \mathcal{H}}\right\}\|(\Lambda + \lambda I)^{-\frac{1}{2}}\Lambda^s\|_{L^2 \to L^2}\\
        \leq & 2 \|(\hat{A}_{n}+\lambda I)^{\frac{1}{2}}(I-VV^*)(\hat{A}_{n}+\lambda I)^{\frac{1}{2}}\|_{\mathcal{H} \to \mathcal{H}}\|(\Lambda + \lambda I)^{-\frac{1}{2}}\Lambda^s\|_{L^2 \to L^2} \quad (\text{by Lemma \ref{lessthan1}})\\
        \lesssim_{p} &  \lambda \|(\Lambda + \lambda I)^{-\frac{1}{2}}\Lambda^s\|_{L^2 \to L^2},
    \end{split}
\end{equation}
where we have used the fact $\|AA^*\| = \|A^*A\| = \|A\|^2$ and Lemma \ref{J T} in $(*)$. Lemma \ref{powersbound} and the fact that $V(V^*\hat{A}_{n}V+\lambda I_{t})^{-1}V^*(\hat{A}_{n}+\lambda I)VV^* = VV^*$ has been used for $(\dag)$. The last step follows from Lemma \ref{projbound}.\\

\noindent
Now as $\|(\Lambda + \lambda I)^{-\frac{1}{2}}\Lambda^s\|_{L^2 \to L^2} = \|(\Lambda + \lambda I)^{-(\frac{1}{2}-s)}(\Lambda + \lambda I)^{-s}\Lambda^s\|_{L^2 \to L^2} \leq \lambda^{s-\frac{1}{2}}$, we have
\begin{equation*}
    \|C^{\frac{1}{2}}J(V(V^*\hat{A}_{n}V+\lambda I_{t})^{-1}V^*-(\hat{A}_{n}+\lambda I)^{-1})J^*\hat{C}_{n}\beta^*\|_{L^2} \lesssim_{p} \lambda^{s+\frac{1}{2}}.
\end{equation*}
\vspace{1cm}
\noindent
Estimation of \textit{Term-2b}: Using Lemma \ref{J T}, we see that
\begin{equation*}
    \begin{split}
        & \|C^{\frac{1}{2}}J(\hat{A}_{n}+\lambda I)^{-1}J^*\hat{C}_{n}\beta^*-C^{\frac{1}{2}}\beta^*\|_{L^2} \leq  \|C^{\frac{1}{2}}J(A+\lambda I)^{-1}J^*C\beta^*-C^{\frac{1}{2}}\beta^*\|_{L^2} \\
        & \qquad \qquad \qquad + \|C^{\frac{1}{2}}J(\hat{A}_{n}+\lambda I)^{-1}J^*\hat{C}_{n}\beta^*- C^{\frac{1}{2}}J(A+\lambda I)^{-1}J^*C\beta^*\|_{L^2}\\
        \leq & \underbrace{\|C^{\frac{1}{2}}T^{\frac{1}{2}}(\hat{\Lambda}_{n}+\lambda I)^{-1}T^{\frac{1}{2}}\hat{C}_{n}\beta^*- C^{\frac{1}{2}}T^{\frac{1}{2}}(\Lambda+\lambda I)^{-1}T^{\frac{1}{2}}C\beta^*\|_{L^2}}_{\textit{Term-2b(1)}}\\
        & \qquad \qquad \qquad + \underbrace{\|C^{\frac{1}{2}}T^{\frac{1}{2}}(\Lambda+\lambda I)^{-1}T^{\frac{1}{2}}C\beta^*-C^{\frac{1}{2}}\beta^*\|_{L^2}}_{\textit{Term-2b(2)}}.
    \end{split}
\end{equation*}

\noindent
Bound of \textit{Term-2b(1)}: Let $\tilde{\beta}:= (\Lambda+\lambda I)^{-1}T^{\frac{1}{2}}C\beta^*$.
\begin{equation}
\label{empirical diff}
    \begin{split}
        & \|C^{\frac{1}{2}}T^{\frac{1}{2}}(\hat{\Lambda}_{n}+\lambda I)^{-1}T^{\frac{1}{2}}\hat{C}_{n}\beta^*- C^{\frac{1}{2}}T^{\frac{1}{2}}(\Lambda+\lambda I)^{-1}T^{\frac{1}{2}}C\beta^*\|_{L^2}\\
        = & \|\Lambda^{\frac{1}{2}}(\hat{\Lambda}_{n}+\lambda I)^{-1}(T^{\frac{1}{2}}\hat{C}_{n}\beta^*- \hat{\Lambda}_{n}\tilde{\beta}-\lambda \tilde{\beta})\|_{L^2}\\
        \lesssim_{p} & \|(\hat{\Lambda}_{n}+\lambda I)^{-\frac{1}{2}}(T^{\frac{1}{2}}\hat{C}_{n}\beta^*- \hat{\Lambda}_{n}\tilde{\beta}-\lambda \tilde{\beta})\|_{L^2} \quad (\text{by Lemma } \ref{powersbound}) \\
        = & \|(\hat{\Lambda}_{n}+\lambda I)^{-\frac{1}{2}}(T^{\frac{1}{2}}\hat{C}_{n}\beta^*- \hat{\Lambda}_{n}\tilde{\beta}-T^{\frac{1}{2}}C\beta^*+\Lambda \tilde{\beta})\|_{L^2} \quad (\because \lambda \tilde{\beta} = T^{\frac{1}{2}}C\beta^* - \Lambda \tilde{\beta})\\
        = & \|(\hat{\Lambda}_{n}+\lambda I)^{-\frac{1}{2}}T^{\frac{1}{2}}(C-\hat{C}_{n})(T^{\frac{1}{2}}\tilde{\beta}-\beta^*)\|_{L^2}\\
        \stackrel{(*)}{=} & \lambda \|(\hat{\Lambda}_{n}+\lambda I)^{-\frac{1}{2}}(\Lambda-\hat{\Lambda}_{n})(\Lambda+\lambda I)^{-1}\Lambda^s h\|_{L^2}\\
        \lesssim & \lambda \|(\hat{\Lambda}_{n}+\lambda I)^{-\frac{1}{2}}(\Lambda-\hat{\Lambda}_{n})(\Lambda+\lambda I)^{-\frac{1}{2}}\|_{L^2 \to L^2} \|(\Lambda+\lambda I)^{-\frac{1}{2}}\Lambda^s\|_{L^2 \to L^2},
    \end{split}
\end{equation}
where $(*)$ follows from $T^{\frac{1}{2}}\tilde{\beta}-\beta^* = -\lambda T^{\frac{1}{2}}(\Lambda +\lambda I)^{-1}\Lambda^sh$.
Next by using Lemma $\ref{ncmdifference}$, we get
\begin{equation*}
    \|C^{\frac{1}{2}}T^{\frac{1}{2}}(\hat{\Lambda}_{n}+\lambda I)^{-1}T^{\frac{1}{2}}\hat{C}_{n}\beta^*- C^{\frac{1}{2}}T^{\frac{1}{2}}(\Lambda+\lambda I)^{-1}T^{\frac{1}{2}}C\beta^*\|_{L^2} \lesssim_{p} \lambda^{s+\frac{1}{2}} \frac{\lambda^{-\frac{1}{b}}}{\sqrt{n}}.
\end{equation*}
Bound of \textit{Term-2b(2)}:
\begin{equation*}
    \begin{split}
        & \|C^{\frac{1}{2}}T^{\frac{1}{2}}(\Lambda+\lambda I)^{-1}T^{\frac{1}{2}}C\beta^*-C^{\frac{1}{2}}\beta^*\|_{L^2}\\
        = & \|C^{\frac{1}{2}}T^{\frac{1}{2}}(\Lambda+\lambda I)^{-1}T^{\frac{1}{2}}CT^{\frac{1}{2}}\Lambda^sh-C^{\frac{1}{2}}T^{\frac{1}{2}}\Lambda^sh\|_{L^2} \\
        = & \|\Lambda^{\frac{1}{2}}((\Lambda+\lambda I)^{-1}\Lambda-I) \Lambda^sh\|_{L^2} \lesssim  \lambda^{s+\frac{1}{2}},
    \end{split}
\end{equation*}
where the last step is obtained by using the qualification property of the Tikhonov regularization. Combining all the estimates, we get
\begin{equation*}
    \|C^{\frac{1}{2}}(\beta^*-\hat{\beta}_{m})\|_{L^2} \lesssim_{p} \frac{\lambda^{-\frac{1}{2b}}}{\sqrt{n}} + \lambda^{s+\frac{1}{2}}.
\end{equation*}
So for $\lambda = n^{-\frac{b}{1+b+2sb}}$, we get
\begin{equation*}
    \|C^{\frac{1}{2}}(\beta^*-\hat{\beta}_{m})\|_{L^2} \lesssim_{p} n^{-\frac{b(1+2s)}{2(1+b+2sb)}}.
\end{equation*}
\end{proof}


\subsection{Proof of part \ref{b}}
\begin{proof}
\begin{equation*}
\begin{split}
    \|\beta^*-\hat{\beta}_{m}\|_{\mathcal{H}}  = &  \|V(V^*\hat{A}_{n}V +\lambda I_{t})^{-1}V^*J^*\hat{R}-\beta^*\|_{\mathcal{H}}\\
     \leq &  \underbrace{\|V(V^*\hat{A}_{n}V+\lambda I_{t})^{-1}V^*(J^*\hat{R}-J^*\hat{C}_{n}\beta^*)\|_{\mathcal{H}}}_{\textit{Term-1}} \\
        & + \underbrace{\|V(V^*\hat{A}_{n}V+\lambda I_{t})^{-1}V^*J^*\hat{C}_{n}\beta^*-\beta^*\|_{\mathcal{H}}}_{\textit{Term-2}}.
\end{split}
    \end{equation*}
\noindent
\textit{Bound for Term-1:}
\begin{equation*}
    \begin{split}
       & \|V(V^*\hat{A}_{n}V+\lambda I_{t})^{-1}V^*(J^*\hat{R}-J^*\hat{C}_{n}\beta^*)\|_{\mathcal{H}}\\
       \leq & \|(\hat{A}_{n}+\lambda I)^{-\frac{1}{2}}(\hat{A}_{n}+\lambda I)^{\frac{1}{2}}V(V^*\hat{A}_{n}V+\lambda I_{t})^{-1}V^*(\hat{A}_{n}+\lambda I)^{\frac{1}{2}}\|_{\mathcal{H} \to \mathcal{H}}\\
       & \times \|(\hat{A}_{n}+\lambda I)^{-\frac{1}{2}}(A+\lambda I)^{\frac{1}{2}}\|_{\mathcal{H} \to \mathcal{H}} \|(A+\lambda I)^{-\frac{1}{2}}(J^*\hat{R}-J^*\hat{C}_{n}\beta^*)\|_{\mathcal{H}}\\
       \lesssim_{p} & \frac{1}{\sqrt{\lambda}} \sqrt{\frac{\sigma^2 \mathcal{N}(\lambda )}{n \delta}},
    \end{split}
\end{equation*}
where last step follows from the combined use of Lemma \ref{lessthan1}, Lemma \ref{powersbound} and Lemma \ref{Empirical bound}.

\vspace{0.5cm}
\noindent
\textit{Bound for Term-2:}
\begin{equation*}
    \begin{split}
        & \|V(V^*\hat{A}_{n}V+\lambda I_{t})^{-1}V^*J^*\hat{C}_{n}\beta^*-\beta^*\|_{\mathcal{H}}\\
        \leq & \underbrace{\|(V(V^*\hat{A}_{n}V+\lambda I_{t})^{-1}V^*-(\hat{A}_{n}+\lambda I)^{-1})J^*\hat{C}_{n}\beta^*\|_{\mathcal{H}}}_{\textit{Term-2a}} + \underbrace{\|(\hat{A}_{n}+\lambda I)^{-1}J^*\hat{C}_{n}\beta^*-\beta^*\|_{\mathcal{H}}}_{\textit{Term-2b}}.
    \end{split}
\end{equation*}
\textit{Estimation of Term-2a:}
Following the similar steps of equation $(\ref{term2a})$, we get
\begin{equation*}
    \begin{split}
       & \|(V(V^*\hat{A}_{n}V+\lambda I_{t})^{-1}V^*-(\hat{A}_{n}+\lambda I)^{-1})J^*\hat{C}_{n}\beta^*\|_{\mathcal{H}}\\
       \lesssim & \|(V(V^*\hat{A}_{n}V+\lambda I_{t})^{-1}V^*-(\hat{A}_{n}+\lambda I)^{-1})J^*\hat{C}_{n}T^{\frac{1}{2}}(\Lambda + \lambda I)^{\frac{1}{2}}\|_{L^2 \to \mathcal{H}} \|(\Lambda + \lambda I)^{-\frac{1}{2}}\Lambda^s\|_{L^2 \to L^2}\\
       \lesssim_{p} & \|(\hat{A}_{n}+\lambda I)^{-\frac{1}{2}}((\hat{A}_{n}+\lambda I)^{\frac{1}{2}}V(V^*\hat{A}_{n}V+\lambda I_{t})^{-1}V^*(\hat{A}_{n}+\lambda I)^{\frac{1}{2}}-I)(\hat{A}_{n}+\lambda I)\|_{\mathcal{H} \to \mathcal{H}}\\
       & \qquad \qquad \qquad \times  \|(\Lambda + \lambda I)^{-\frac{1}{2}}\Lambda^s\|_{L^2 \to L^2}\\
        \leq & \lambda^{s-1}\|(\hat{A}_{n}+\lambda I)^{\frac{1}{2}}V(V^*\hat{A}_{n}V+\lambda I_{t})^{-1}V^*(\hat{A}_{n}+\lambda I) - (\hat{A}_{n}+\lambda I)^{\frac{1}{2}} ) (\hat{A}_{n}+\lambda I)^{\frac{1}{2}}\|_{\mathcal{H} \to \mathcal{H}}\\
        = & \lambda^{s-1} \|(\hat{A}_{n}+\lambda I)^{\frac{1}{2}}[I- VV^* +V(V^*\hat{A}_{n}V+\lambda I_{t})^{-1}V^*(\hat{A}_{n}+\lambda I)(VV^*-I)](\hat{A}_{n}+\lambda I)^{\frac{1}{2}}\|_{\mathcal{H} \to \mathcal{H}}\\
        \leq & \lambda^{s-1}\|(\hat{A}_{n}+\lambda I)^{\frac{1}{2}}V(V^*\hat{A}_{n}V+\lambda I_{t})^{-1}V^*(\hat{A}_{n}+\lambda I) (VV^* - I) (\hat{A}_{n}+\lambda I)^{\frac{1}{2}}\|_{\mathcal{H} \to \mathcal{H}}\\
        & + \lambda^{s-1}\|(\hat{A}_{n}+\lambda I)^{\frac{1}{2}}(I- VV^*)(\hat{A}_{n}+\lambda I)^{\frac{1}{2}}\|_{\mathcal{H} \to \mathcal{H}}\\
        \leq & \lambda^{s-1}\|(\hat{A}_{n}+\lambda I)^{\frac{1}{2}}(I- VV^*)(\hat{A}_{n}+\lambda I)^{\frac{1}{2}}\|_{\mathcal{H} \to \mathcal{H}} \lesssim_{p} \lambda^s,
    \end{split}
\end{equation*}
where the last step follows from Lemma \ref{lessthan1} and Lemma \ref{projbound}.\\

\noindent
\textit{Estimation of Term-2b:}
\begin{equation*}
    \begin{split}
       & \|(\hat{A}_{n}+\lambda I)^{-1}J^*\hat{C}_{n}\beta^*-\beta^*\|_{\mathcal{H}} = \|J(\hat{A}_{n}+\lambda I)^{-1}J^*\hat{C}_{n}\beta^*-\beta^*\|_{\mathcal{H}}\\
       \leq & \underbrace{\|J(A+\lambda I)^{-1}J^*C\beta^*-\beta^*\|_{\mathcal{H}}}_{Term-2b(1)} + \underbrace{\|J(\hat{A}_{n}+\lambda I)^{-1}J^*\hat{C}_{n}\beta^*- J(A+\lambda I)^{-1}J^*C\beta^*\|_{\mathcal{H}}}_{Term-2b(2)}. 
    \end{split}
\end{equation*}

\noindent
\textit{Bound for Term-2b(1):}
\begin{equation*}
    \begin{split}
        \|J(A+\lambda I)^{-1}J^*C\beta^*-\beta^*\|_{\mathcal{H}} \stackrel{(*)}{=} & \|T^{\frac{1}{2}}(\Lambda+\lambda I)^{-1}T^{\frac{1}{2}}C\beta^*-\beta^*\|_{\mathcal{H}}\\
        = & \|T^{\frac{1}{2}}((\Lambda+\lambda I)^{-1}T^{\frac{1}{2}}CT^{\frac{1}{2}}-I)\Lambda^sh\|_{\mathcal{H}}\\
        \stackrel{(\dag)}{\lesssim} & \lambda^s,
    \end{split}
\end{equation*}
where $(*)$ follows from Lemma \ref{J T} and for $(\dag)$, we have used qualification property of the Tikhonov regularization.\\

\noindent
\textit{Bound for Term-2b(2):} By Lemma \ref{J T}, we have
\begin{equation*}
    \begin{split}
        & \|J(\hat{A}_{n}+\lambda I)^{-1}J^*\hat{C}_{n}\beta^*- J(A+\lambda I)^{-1}J^*C\beta^*\|_{\mathcal{H}}\\
        &  \qquad \qquad = \|T^{\frac{1}{2}}(\hat{\Lambda}_{n}+\lambda I)^{-1}T^{\frac{1}{2}}\hat{C}_{n}\beta^*- T^{\frac{1}{2}}(\Lambda+\lambda I)^{-1}T^{\frac{1}{2}}C\beta^*\|_{\mathcal{H}}\\
        &  \qquad \qquad \leq \|(\hat{\Lambda}_{n}+\lambda I)^{-1}T^{\frac{1}{2}}\hat{C}_{n}\beta^*- (\Lambda+\lambda I)^{-1}T^{\frac{1}{2}}C\beta^*\|_{L^2}.
    \end{split}
\end{equation*}
Following the similar steps of equation $(\ref{empirical diff})$, we get
\begin{equation*}
\begin{split}
    & \|J(\hat{A}_{n}+\lambda I)^{-1}J^*\hat{C}_{n}\beta^*- J(A+\lambda I)^{-1}J^*C\beta^*\|_{\mathcal{H}}\\
\leq & \sqrt{\lambda} \|(\hat{\Lambda}_{n}+\lambda I)^{-\frac{1}{2}}(\Lambda-\hat{\Lambda}_{n})(\Lambda+\lambda I)^{-1}\Lambda^s h\|_{L^2}
    \lesssim_{p} \lambda^s \frac{\lambda^{-\frac{1}{b}}}{\sqrt{n}}.
\end{split}
\end{equation*}
Combining all the bounds, we conclude
\begin{equation*}
    \|\beta^*-\hat{\beta}_{m}\|_{\mathcal{H}} \lesssim_{p} \frac{\lambda^{-\frac{1}{2b}-\frac{1}{2}}}{\sqrt{n}} + \lambda^s.
\end{equation*}
So for $\lambda = n^{-\frac{b}{1+b+2sb}}$, we get
\begin{equation*}
    \|\beta^*-\hat{\beta}_{m}\|_{\mathcal{H}} \lesssim_{p} n^{-\frac{bs}{1+b+2sb}}.
\end{equation*}
\end{proof}
\begin{remark}\label{combinedrmk} In the proof of Theorem~\ref{mainresult}, we primarily relied on part $(i)$ of Assumption \ref{as:3}, especially through Lemma~\ref{ncmdifference}. However, a similar proof could also be constructed using part $(ii)$ of Assumption \ref{as:3}, with Lemma \ref{ncmdiffnew} replacing Lemma \ref{ncmdifference}. Researchers have expressed differing preferences regarding which condition is most suitable for the FLR model. While some have favored part $(i)$ \cite{ARKHSFORFLR, tonyyuan2012minimax, ZhangFaster2020, guo2023capacity}, others have advocated for part $(ii)$ \cite{HTONGFLR, TONGHUBER, polynomial2023regularization}. As our convergence rates matches with the existing lower bounds for the FLR model \cite[Theorem $2.2$]{ZhangFaster2020}, this article demonstrates that optimal convergence rates can still be achieved under any one of the assumptions. \end{remark}
\section{Supplementary Results}
\label{supplements}
In this section, we present technical results that are used to prove the main results of the paper.\\
The following lemma is a simple analogue of \cite[Lemma 8]{rudi2015less}.
\begin{lemma}
\label{lessthan1}
Let $\lambda > 0$. Suppose $V$  is such that $V^*V = I_{t}$ and $\hat{A}_{n}$ be any positive self-adjoint operator. Then
    \begin{equation*}
        \|(\hat{A}_{n}+\lambda I)^{\frac{1}{2}}V(V^*\hat{A}_{n}V+\lambda I_{t})^{-1}V^*(\hat{A}_{n}+\lambda I)^{\frac{1}{2}}\|_{op} \leq 1.
    \end{equation*}
\end{lemma}
\begin{proof}
    \begin{equation*}
        \begin{split}
            & \|(\hat{A}_{n}+\lambda I)^{\frac{1}{2}}V(V^*\hat{A}_{n}V+\lambda I_{t})^{-1}V^*(\hat{A}_{n}+\lambda I)^{\frac{1}{2}}\|_{op}^2\\ 
            \stackrel{(*)}{=} &\|(\hat{A}_{n}+\lambda I)^{\frac{1}{2}}V(V^*\hat{A}_{n}V+\lambda I_{t})^{-1}V^*(\hat{A}_{n}+\lambda I)V(V^*\hat{A}_{n}V+\lambda I_{t})^{-1}V^*(\hat{A}_{n}+\lambda I)^{\frac{1}{2}}\|_{op}\\
            \stackrel{(\dag)}{=} & \|(\hat{A}_{n}+\lambda I)^{\frac{1}{2}}V(V^*\hat{A}_{n}V+\lambda I_{t})^{-1}V^*(\hat{A}_{n}+\lambda I)^{\frac{1}{2}}\|_{op}.
        \end{split}
    \end{equation*}
Here $(*)$ follows from that fact that $\|A^*A\| = \|A\|^2 $ for any bounded operator $A$ and for $(\dag)$, we have used that $V^*V = I_{t}$.
\end{proof}
\begin{lemma}\cite[Lemma A.5]{gupta2024optimal}
\label{ncmdifference}
    Let Assumptions \ref{as:1},
    \ref{as:2} and \ref{as:3} holds. Then, we have
    \begin{equation*}
    \|(\Lambda+\lambda I)^{-\frac{1}{2}}(\Lambda-\hat{\Lambda}_{n})(\Lambda+\lambda I)^{-\frac{1}{2}}\|_{L^2 \to L^2} \lesssim_{p} \frac{1}{\sqrt{n}} \lambda^{-\frac{1}{b}}.
\end{equation*}
\end{lemma}

\noindent
The next lemma is an analogous of Lemma 6 \cite{rudi2015less}.
\begin{lemma}
\label{projbound}
Assume $\lambda \gtrsim m^{-\frac{1}{b}} \geq n^{-\frac{1}{b}}$. Suppose $\{\tilde{X}_{i}; 1\leq i \leq m\}$ is subsampled uniformly at random from the data $\{X_{i}; 1\leq i \leq n\}$. Then for any $\delta >0$, we have with probability at least $1-\delta$
$$\|(I-P_{m})A^{\frac{1}{2}}\|_{\mathcal{H} \to \mathcal{H}}^2 \leq  \lambda,$$
where $P_{m}$ is the projection operator on $\mathcal{H}_{m} := \textit{span}\left\{\int_{S}k(\cdot, t)\tilde{X}_{i}(t)dt: 1\leq i \leq m\right\}$.
\end{lemma}
\begin{proof}
Let us define $\hat{A}_{m} = J^*\hat{C}_{m}J$, where $\hat{C}_{m} = \frac{1}{m}\sum_{i=1}^{m} \tilde{X}_{i} \otimes \tilde{X}_{i}$, where $\{\tilde{X}_{i} : 1\leq i \leq m \}$ is subsampled data set. Now consider
    \begin{equation*}
        \begin{split}
            \|(I-P_{m})A^{\frac{1}{2}}\|_{\mathcal{H} \to \mathcal{H}} = & \|(I-P_{m})(\hat{A}_{m}+\lambda I)^{\frac{1}{2}}(\hat{A}_{m}+\lambda I)^{-\frac{1}{2}}A^{\frac{1}{2}}\|_{\mathcal{H} \to \mathcal{H}} \\
            \leq & \|(I-P_{m})(\hat{A}_{m}+\lambda I)^{\frac{1}{2}}\|_{\mathcal{H} \to \mathcal{H}}
            \|(\hat{A}_{m}+\lambda I)^{-\frac{1}{2}}A^{\frac{1}{2}}\|_{\mathcal{H} \to \mathcal{H}}\\
            \leq & \|(I-P_{m})(\hat{A}_{m}+\lambda I)^{\frac{1}{2}}\|_{\mathcal{H} \to \mathcal{H}},
        \end{split}
    \end{equation*}
where the last step follows from Lemma \ref{powersbound} applied for $\hat{A}_m$ in place of $\hat{A}_n$.\\ 
Next, we consider
\begin{equation*}
    \begin{split}
        \langle f, \hat{A}_{m}^{\frac{1}{2}}(\hat{A}_{m}+\lambda I)^{-1}\hat{A}_{m}^{\frac{1}{2}} f \rangle_{\mathcal{H}} = & \langle f, P_{m}\hat{A}_{m}^{\frac{1}{2}}(\hat{A}_{m}+\lambda I)^{-1}\hat{A}_{m}^{\frac{1}{2}} P_{m}f \rangle_{\mathcal{H}} \\
        = & \|(\hat{A}_{m}+\lambda I)^{-\frac{1}{2}}\hat{A}_{m}^{\frac{1}{2}} P_{m}f\|_{\mathcal{H}}^2 \leq \langle f, P_{m}f \rangle_{\mathcal{H}}.
    \end{split}
\end{equation*}
So we have that $P_{m}-\hat{A}_{m}^{\frac{1}{2}}(\hat{A}_{m}+\lambda I)^{-1}\hat{A}_{m}^{\frac{1}{2}} \geq 0$, which implies $(I-\hat{A}_{m}^{\frac{1}{2}}(\hat{A}_{m}+\lambda I)^{-1}\hat{A}_{m}^{\frac{1}{2}}) - (I-P_{m}) \geq 0$. As we can observe that 
$(I-\hat{A}_{m}^{\frac{1}{2}}(\hat{A}_{m}+\lambda I)^{-1}\hat{A}_{m}^{\frac{1}{2}}) = \lambda (\hat{A}_{m}+\lambda I)^{-1} $, we have
$\lambda (\hat{A}_{m}+\lambda I)^{-1} - (I-P_{m})\geq 0$.\\

\noindent
Now applying proposition $5$ \cite{rudi2015less}, we get 
\begin{equation*}
    \|(I-P_{m})A^{\frac{1}{2}}\|_{\mathcal{H} \to \mathcal{H}} \leq \lambda \|(\hat{A}_{m}+\lambda I)^{-1} (\hat{A}_{m}+\lambda I)^{\frac{1}{2}}\|_{\mathcal{H} \to \mathcal{H}} \leq \sqrt{\lambda}.
\end{equation*}

\end{proof}
\begin{lemma}\cite[Lemma 5.1]{cordes1987}(Corde's Inequality)
\label{cordes}
Suppose $T_1$ and $T_2$ are two positive bounded linear operators on a separable Hilbert space. Then
$$\|T_1^pT_2^p\| \leq \|T_1T_2\|^p, \text{ when } 0\leq p \leq 1. $$
\end{lemma}

\begin{lemma}
\label{powersbound}
Let $\lambda \geq \left(\frac{4c_{1}^2}{n}\right)^{\frac{b}{b+1}}$ for some positive constant $c_{1}$. Under Assumptions \ref{as:2} and \ref{as:3}, with probability at least $1-\delta$, we get
\begin{equation*}
\begin{split}
    \|(A+\lambda I)^{\frac{1}{2}}(\hat{A}_{n}+\lambda I)^{-\frac{1}{2}}\|_{\mathcal{H} \to \mathcal{H}} \leq \sqrt{2}\quad  \text{and} \quad
    \|(\hat{A}_{n}+\lambda I)^{\frac{1}{2}}(A+\lambda I)^{-\frac{1}{2}}\|_{\mathcal{H} \to \mathcal{H}} \leq  \sqrt{\frac{3}{2}}.
    \end{split}
\end{equation*} 
\end{lemma}
\begin{proof}
By the fact that $\|B^*B\| = \|BB^*\| = \|B\|^2$ for any bounded operator $B$, we have
    \begin{equation*}
    \begin{split}
        \|(A+\lambda I)^{\frac{1}{2}}(\hat{A}_{n}+\lambda I)^{-\frac{1}{2}}\|_{\mathcal{H} \to \mathcal{H}}^2 = & \|(\hat{A}_{n}+\lambda I)^{-\frac{1}{2}}(A+\lambda I)(\hat{A}_{n}+\lambda I)^{-\frac{1}{2}}\|_{\mathcal{H} \to \mathcal{H}}\\
        = & \|(\hat{A}_{n}+\lambda I)^{-\frac{1}{2}}(A+\lambda I)^{\frac{1}{2}}\|^2_{\mathcal{H} \to \mathcal{H}}.
        \end{split}
    \end{equation*}
Assuming that $\|(\hat{A}_{n}-A)(A + \lambda I)^{-1}\|_{\mathcal{H} \to \mathcal{H}} \leq \frac{1}{2}$, we have
\begin{equation*}
    \begin{split}
        \|(A+\lambda I)(\hat{A}_{n}+\lambda I)^{-1}\|_{\mathcal{H} \to \mathcal{H}} = & \|[(\hat{A}_{n}-A)(A + \lambda I)^{-1}+I]^{-1}\|_{\mathcal{H} \to \mathcal{H}}\\
        \leq & \frac{1}{1-\|(\hat{A}_{n}-A)(A + \lambda I)^{-1}\|_{\mathcal{H} \to \mathcal{H}}}.
    \end{split}
\end{equation*}
using Lemma \ref{J T}, we can see that $\|(\hat{A}_{n}-A)(A + \lambda I)^{-1}\|_{\mathcal{H} \to \mathcal{H}} = \|(\hat{\Lambda}_{n}-\Lambda)(\Lambda + \lambda I)^{-1}\|_{L^2 \to L^2}$ and following the similar steps of Lemma $A.5$ \cite{gupta2024optimal}, we get
\begin{equation*}
    \|(\hat{A}_{n}-A)(A + \lambda I)^{-1}\|_{\mathcal{H} \to \mathcal{H}} \leq c_{1}\frac{1}{\sqrt{n}} \lambda^{-\frac{1}{2b}-\frac{1}{2}},
\end{equation*}
then for $\lambda \geq \left(\frac{4c_{1}^2}{n}\right)^{\frac{b}{b+1}}$ , where $c_{1}$ is a positive constant,
$$\|(\hat{A}_{n}-A)(A + \lambda I)^{-1}\|_{\mathcal{H} \to \mathcal{H}} \leq \frac{1}{2}.$$
Using this, we establish that
\begin{equation*}
    \|(A+\lambda I)(\hat{A}_{n}+\lambda I)^{-1}\|_{\mathcal{H} \to \mathcal{H}} \leq 2 \quad \forall~ \lambda \geq \left(\frac{4c_{1}^2}{n}\right)^{\frac{b}{b+1}}.
\end{equation*}
Now by using Lemma \ref{cordes}, we get that
\begin{equation*}
    \|(A+\lambda I)^{\frac{1}{2}}(\hat{A}_{n}+\lambda I)^{-\frac{1}{2}}\|_{\mathcal{H} \to \mathcal{H}} \leq \sqrt{2} \quad \forall~ \lambda \geq \left(\frac{4c_{1}^2}{n}\right)^{\frac{b}{b+1}}.
\end{equation*}
Note that 
\begin{equation*}
    \begin{split}
        \|(A+\lambda I)^{-1}(\hat{A}_{n}+\lambda I)\|_{\mathcal{H} \to \mathcal{H}} = & \|I - (A- \hat{A}_{n})(A+\lambda I)^{-1}\|_{\mathcal{H} \to \mathcal{H}}\\
        \leq & 1 + \|(\hat{A}_{n}-A)(A+\lambda I)^{-1}\|_{\mathcal{H} \to \mathcal{H}}\\
        \leq & \frac{3}{2} \quad \forall~ \lambda \geq \left(\frac{4c_{1}^2}{n}\right)^{\frac{b}{b+1}},
    \end{split}
\end{equation*}
then again by Lemma \ref{cordes}, we get
\begin{equation*}
    \|(A+\lambda I)^{-\frac{1}{2}}(\hat{A}_{n}+\lambda I)^{\frac{1}{2}}\|_{\mathcal{H} \to \mathcal{H}} \leq \sqrt{\frac{3}{2}} \quad \forall~ \lambda \geq \left(\frac{4c_{1}^2}{n}\right)^{\frac{b}{b+1}}.
\end{equation*}
\end{proof}
\begin{lemma}
\label{ncmdiffnew}
    Suppose Assumptions \ref{as:1} and
    \ref{as:2} hold. Also assume that $\sup_{\omega \in \Omega}\|X(\cdot, \omega)\|_{L^2} < \infty$. Then for $r \geq \frac{1}{2}$, we have
    \begin{equation*}
    \|(\Lambda+\lambda I)^{-r}(\Lambda-\hat{\Lambda}_{n})(\Lambda+\lambda I)^{-r}\|_{L^2 \to L^2} \lesssim_{p} \frac{\lambda^{-2r}}{n} + \frac{1}{\sqrt{n}} \lambda^{-\frac{1+4br - b}{2b}}.
\end{equation*}
\end{lemma}
The proof of Lemma \ref{ncmdiffnew} can be straightforwardly derived by applying \cite[Theorem 6.14]{SVM2008steinwart} with $\xi = (\Lambda+\lambda I)^{-r} T^{\frac{1}{2}}X \otimes T^{\frac{1}{2}}X$.
\begin{lemma} \cite[Lemma 4.2]{gupta2023convergence}
\label{Empirical bound}
For $\delta >0$, with at least probability $1-\delta$, we have that
\begin{equation*}
\|(A+\lambda I)^{-\frac{1}{2}}J^*(\hat{R}-\hat{C}_{n}\beta^*)\|_{\mathcal{H}} \leq \sqrt{\frac{ \sigma^2 \mathcal{N}(\lambda)}{n \delta}}.
\end{equation*}
\end{lemma}
\begin{lemma}\cite[Lemma A.1]{gupta2023convergence}
\label{J T}
    For any bounded linear positive operator $\mathcal{A}:L^2(S)\to L^2(S)$, we have
\begin{equation*}
Jg_{\lambda}(J^*\mathcal{A}J)J^* = T^{1/2}g_{\lambda}(T^{\frac{1}{2}}\mathcal{A}T^{\frac{1}{2}})T^{1/2}.
\end{equation*}
\end{lemma}

\begin{lemma}\cite[Lemma A.10]{gupta2024optimal}
\label{seriessum}
    For $\alpha >1$, $\beta >1,$ and $q \geq \frac{\alpha}{\beta}$, we have 
    $$\sum_{i\in \mathbb{N}}\frac{i^{-\alpha}}{(i^{-\beta}+\lambda)^q} \lesssim \lambda^{-\frac{1+\beta q -\alpha}{\beta}}, ~~ \forall ~\lambda >0.$$
\end{lemma}

\noindent

\section{Numerical Experiment}
\label{numerical}
In this section, we illustrate the proposed Nyström subsampling method with an example. Let $S = [0,1]$ and consider the reproducing kernel Hilbert space:
$$H_{K} := \left\{f(x) = \sum_{k \geq 1} f_{k} \sqrt{2} \cos(k \pi x) ~ |~ \sum_{k \geq 1} k^4 f_{k}^2 < \infty \right\} $$
with inner product 
\begin{equation*}
    \langle f, g  \rangle_{H_{k}} := \sum_{k\geq 1} (k \pi)^4 f_{k} g_{k}.
\end{equation*}
It can be seen from \cite{tonyyuan2012minimax} that the associated reproducing kernel of $H_k$ is given as:
\begin{equation*}
    K(s,t) = \sum_{k \geq 1} \frac{2}{(k \pi)^4} \cos(k \pi s) \cos(k \pi t) = -\frac{1}{3}\left(B_{4}\left(\frac{s+t}{2}\right)+ B_{4}\left(\frac{|s-t|}{2}\right)\right),
\end{equation*}
where $B_{4}$ is the Bernoulli polynomial of degree $4$.
We consider the Brownian motion covariance kernel 
\begin{equation*}
    \begin{split}
        C(s,t ) = & \sum_{n \geq 1} \frac{2}{(n-\frac{1}{2})^2 \pi^2} \sin\left(\left(n-\frac{1}{2}\right)\pi s\right) \sin\left(\left(n-\frac{1}{2}\right)\pi t\right)\\
        = & E_{1}\left(\frac{s+t}{2}\right) - E_{1}\left(\frac{|s-t|}{2}\right) = \min\{s,t\},
    \end{split}
\end{equation*}
where $E_{1}(x) = x- \frac{1}{2}$ is the Euler polynomial of degree $1$. We generate $650$ data points from $X_{i}(t) = \sum_{i = 1}^{500} \frac{\sqrt{2}}{(i-\frac{1}{2})\pi} Z_{i} \sin((i-\frac{1}{2})\pi t)$, $\beta^*(t) = \sum_{k \geq 1} 4\sqrt{2}(-1)^{k-1}k^{-2} \cos(k\pi t) = -\sqrt{2} \pi^2 (E_{1}(t)+B_{2}(t))$ and the model $(\ref{model})$, where $B_{2}(x) = x^2 - x + \frac{1}{6}$ is the Bernoulli polynomial of degree 2 and the coefficients $Z_{i}$ follows a normal distribution with zero mean and variance 1. We have chosen the error term to follow a normal distribution with mean $0$ and $\sigma = \sqrt{0.5}$. We use initial $550$ data points to construct our estimator and remaining $100$ points to measure the prediction error using the root mean squared error (RMSE) which is defined as $\left(\sum_{i=1}^{100}\frac{(\langle\beta^*, X_{500+i}\rangle_{L^2}-\langle \hat{\beta}, X_{500+i}\rangle_{L^2})^2}{100}\right)^{\frac{1}{2}}$. For the illustration, we constructed $\hat{\beta}$ with different values of the regularization parameter $\lambda \in [10^{-07}, 10^{-04}]$ and subsampling size $m \in [10,240]$ and plotted the RMSE averaged over 100 repetitions in Figure \ref{fig:mean_error_heatmap}.\\


\begin{figure}[h!]
    \centering
    \includegraphics[width=0.8\textwidth]{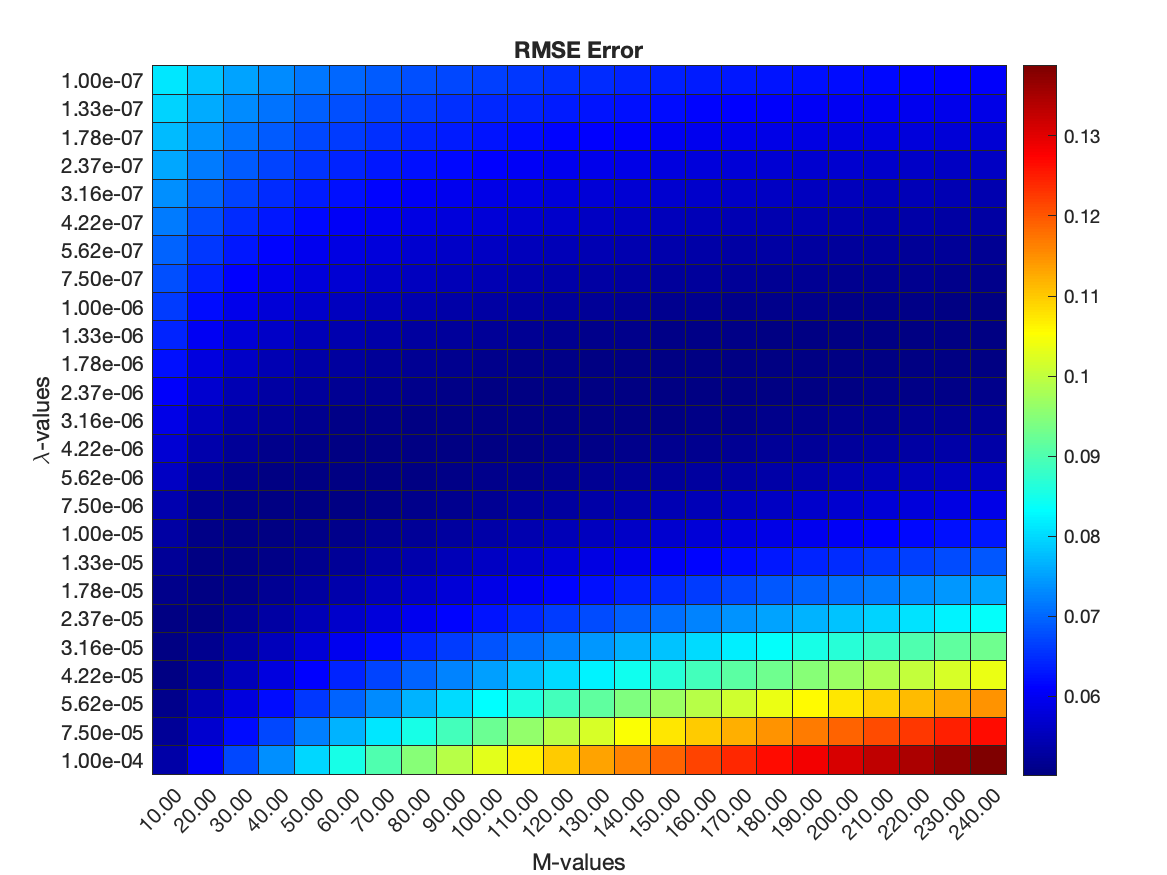}
    \caption{RMSE error associated to 25 $\times$ 25 grids of values for $m~(x\text{-axis })$ and $\lambda~(y\text{-axis})$.} 
    \label{fig:mean_error_heatmap}
\end{figure}
The root mean square error (RMSE) for the kernel method, applied without any subsampling, is found to be 0.06. As observed in Figure~\ref{fig:mean_error_heatmap}, the same level of accuracy (i.e., an RMSE of 0.06) can be achieved for both small and large values of the subsampling parameter $m$, as long as an appropriate regularization parameter is selected. This demonstrates that, with proper regularization, it is possible to maintain high accuracy across different choices of $m$, avoiding significant loss of performance even in cases where subsampling is applied.


\begin{remark}[Concluding Remarks]
We conclude that by using the Nyström subsampling technique in the RKHS framework for the FLR model, we can reduce the computational complexity of the Kernel methods from $O(n^3)$ to $O(m^2n)$, where $n$ is the size of the dataset and $m$ is the size of the subsampled dataset. Compared to the convergence rates of kernel method for the FLR model \cite{tonyyuan2012minimax, ZhangFaster2020, balasubramanian2024functional}, Theorem \ref{mainresult} ensures that while reducing the computational complexity with the help of Nyström subsampling simultaneously we can maintain the accuracy of the kernel methods. Moreover, the convergence rates achieved in Theorem \ref{mainresult} are optimal as they match with the existing lower bounds  \cite[Theorem $2.2$]{ZhangFaster2020}.
\end{remark}
\section*{Acknowledgements}
S. Sivananthan acknowledges the Anusandhan National Research Foundation, Government of India, for the financial support through project no. MTR/2022/000383.
\bibliographystyle{acm}
\bibliography{123}
\end{document}